\documentclass[a4paper]{amsart}
\usepackage{etex}
\usepackage[OT2, T1]{fontenc}
\usepackage{url}
\usepackage{amsmath}
\usepackage{array}
\usepackage{graphicx}
\usepackage{amsfonts}
\usepackage{amssymb}
\usepackage[colorlinks=true,citecolor=blue,linkcolor=red,urlcolor=link]{hyperref}
\usepackage{amsthm}
\usepackage{tikz-cd}
\tikzset{
    labl/.style={anchor=south, rotate=90, inner sep=.5mm}
}
\usepackage{esint}
\usepackage{color}
\usepackage{xcolor}
\definecolor{link}{RGB}{11,0,128}

\usepackage{paralist}
\usepackage{colonequals}		% for nice :=
\usepackage{mathabx}		% for \widec
\usepackage[left]{lineno}
\usepackage{amscd}
\usepackage[all,cmtip]{xy}	% for commutative diagrams
\usepackage{sseq}			% for spectral sequences
\usepackage{verbatim}
\usepackage{parskip}
\usepackage{microtype}
\usepackage[in]{fullpage}
\usepackage{graphicx}
\usepackage{mathrsfs} 			% for \mathscr (script letters)
\usepackage{bm} 				% for bold Greek letters
\usepackage{cleveref}
\usepackage{enumitem}
\usepackage{moreenum}			% for enumeration via Greek letters
\usepackage{blindtext}
\usepackage{tikz-cd}
\usepackage{tikz}
\usetikzlibrary{bending,matrix,arrows,decorations.pathmorphing,backgrounds,positioning,fit,petri,arrows.meta}
\usepackage{textcomp}
\usepackage{indentfirst}
\usetikzlibrary{arrows}
\usepackage{enumitem}
\tikzset{commutative diagrams/.cd,arrow style=tikz,diagrams={>=latex'}}

\usepackage{filecontents}
\usepackage[alphabetic,lite]{amsrefs} 	% for bibliography
\usepackage{stmaryrd}	% this is for \llbracket and \rrbracket
\usepackage[foot]{amsaddr}
\usepackage{subfiles}
\usepackage{ragged2e}

\DeclareSymbolFont{cyrletters}{OT2}{wncyr}{m}{n}
\DeclareMathSymbol{\Sha}{\mathalpha}{cyrletters}{"58}

\newcommand{\CB}{\mathcal{B}}
\newcommand{\CC}{\mathcal{C}}
\newcommand{\cD}{\mathcal{D}}
\newcommand{\CE}{\mathcal{E}}
\newcommand{\CF}{\mathcal{F}}
\newcommand{\CG}{\mathcal{G}}

\newcommand{\CO}{\mathcal{O}}

\newcommand\numeq[1]%
  {\stackrel{\scriptscriptstyle(\mkern-1.5mu#1\mkern-1.5mu)}{=}}
\newcommand{\ra}{\rightarrow}
\newcommand{\la}{\leftarrow}

\newcommand{\pr}{^{\prime}}

\newcommand{\ce}{\colonequals}
\newcommand{\ov}{\overline}

\renewcommand{\b}{\textbf}

\newcommand{\isom}{\simeq}			% Isomorphic
\newcommand{\isomto}{\overset{\sim}{\longrightarrow}}

\newcommand{\dhom}{R\mathcal{H}om}

\newcommand{\flst}{f_{\ast}}
\newcommand{\fls}{f_{!}}
\newcommand{\fust}{f^{\ast}}
\newcommand{\fus}{f^{!}}
\newcommand{\ilst}{i_{\ast}}

\newcommand{\iust}{i^{\ast}}

\newcommand{\jlst}{j_{\ast}}
\newcommand{\jls}{j_{!}}
\newcommand{\just}{j^{\ast}}

\newcommand{\Modc}{\mathrm{Mod}_c}

\providecommand{\p}[1]{\left(#1\right)}
\providecommand{\up}[1]{{\upshape(}#1{\upshape)}}

\DeclareMathOperator{\Spec}{Spec}		% Spectrum of a ring
\DeclareMathOperator{\id}{id}			% identity
\DeclareMathOperator{\Tr}{Tr}		% Trace
\DeclareMathOperator{\Aut}{Aut}		% The group of automorphisms
\newcommand{\ba}{\begin{aligned}}
\newcommand{\ea}{\end{aligned}}
\newcommand{\be}{\begin{equation}}
\newcommand{\ee}{\end{equation}}
\newcommand{\pf}{\begin{proof}}
\newcommand{\bpf}{\begin{proof}}
\newcommand{\epf}{\end{proof}}
\newcommand{\bthm}{\begin{thm}}
\newcommand{\ethm}{\end{thm}}

\newcommand{\bprop}{\begin{prop}}
\newcommand{\eprop}{\end{prop}}
\newcommand{\bcor}{\begin{cor}}
\newcommand{\ecor}{\end{cor}}

\newcommand{\brem}{\begin{rem}}
\newcommand{\erem}{\end{rem}}

\newcommand{\brems}{\begin{rems} \hfill \begin{enumerate}[label=\b{\thenumberingbase.},ref=\thenumberingbase]}

\newcommand{\erems}{\end{enumerate} \end{rems}}

\newcommand{\begs}{\begin{egs} \hfill \begin{enumerate}[label=\b{\thenumberingbase.},ref=\thenumberingbase]}

\newcommand{\eegs}{\end{enumerate} \end{egs}}

\newcommand{\eremstweak}{\end{enumerate} \end{rems-tweak}}
\newcommand{\eremst}{\end{enumerate} \end{rems-tweak}}
\newcommand{\blem}{\begin{lemma}}
\newcommand{\elem}{\end{lemma}}

\newcommand{\bconj}{\begin{conj}}
\newcommand{\econj}{\end{conj}}
\newcommand{\bprob}{\begin{Problem}}
\newcommand{\eprob}{\end{Problem}}

\newcommand{\bq}{\begin{Q}}
\newcommand{\eq}{\end{Q}}
\newcommand{\benum}{\begin{enumerate}[label={{\upshape(\alph*)}}]}
\newcommand{\benuma}{\begin{enumerate}[label={{\upshape(\arabic*)}}]}
\newcommand{\benumr}{\begin{enumerate}[label={{\upshape(\roman*)}}]}
\newcommand{\eenum}{\end{enumerate}}
\newcommand{\bc}{\begin{comment}}
\newcommand{\ec}{\end{comment}}
\newcommand{\bd}{\begin{defn}}
\newcommand{\ed}{\end{defn}}
\newcommand{\bdt}{\begin{defn-tweak}}
\newcommand{\edt}{\end{defn-tweak}}
\newcommand{\beg}{\begin{eg}}
\newcommand{\eeg}{\end{eg}}

\newcommand{\bcl}{\begin{claim}}
\newcommand{\ecl}{\end{claim}}
\newcommand{\lab}{\label}

\newcommand{\q}{\quad}
\newcommand{\qq}{\quad\quad}

\usepackage{aliascnt}
\newaliascnt{numberingbase}{subsection}

\theoremstyle{plain}
\newtheorem{thm}[numberingbase]{Theorem}
\Crefname{thm}{Theorem}{Theorems}

\Crefname{rethm}{Theorem}{Theorem}
\newtheorem{prop}[numberingbase]{Proposition}
\Crefname{prop}{Proposition}{Propositions}
\newtheorem{Q}[numberingbase]{Question}
\Crefname{Q}{Question}{Questions}
\newtheorem{Problem}[subsection]{Problem}
\Crefname{Problem}{Problem}{Problems}
\newtheorem{conj}[numberingbase]{Conjecture}
\Crefname{conj}{Conjecture}{Conjectures}
\newtheorem{cor}[numberingbase]{Corollary}
\Crefname{cor}{Corollary}{Corollaries}
\newtheorem{lemma}[numberingbase]{Lemma}

\Crefname{subprop}{Proposition}{Propositions}

\Crefname{subcor}{Corollary}{Corollaries}

\Crefname{sublem}{Lemma}{Lemmas}

\theoremstyle{remark}
\newtheorem{claim}[equation]{Claim}
\Crefname{claim}{Claim}{Claims}

\Crefname{subrem}{Remark}{Remarks}

\theoremstyle{definition}
\newtheorem{defn}[numberingbase]{Definition}
\Crefname{defn}{Definition}{Definitions}

\Crefname{conv}{Convention}{Conventions}
\newtheorem{eg}[numberingbase]{Example}
\Crefname{eg}{Example}{Examples}
\newtheorem{rem}[numberingbase]{Remark}
\Crefname{rem}{Remark}{Remarks}
\newtheorem*{rems}{Remarks}
\newtheorem*{egs}{Examples}
\newtheorem{assumption}[numberingbase]{Assumption}
\Crefname{assumption}{Assumption}{Assumptions}

\theoremstyle{plain}
\newtheorem{thm-tweak}[subsection]{Theorem}
\Crefname{thm-tweak}{Theorem}{Theorems}
\newtheorem{lemma-tweak}[subsection]{Lemma}
\Crefname{lemma-tweak}{Lemma}{Lemmas}
\newtheorem{cor-tweak}[subsection]{Corollary}
\Crefname{cor-tweak}{Corollary}{Corollaries}
\newtheorem{prop-tweak}[subsection]{Proposition}
\Crefname{prop-tweak}{Proposition}{Propositions}
\newtheorem{conj-tweak}[subsection]{Conjecture}
\Crefname{conj-tweak}{Conjecture}{Conjectures}

\theoremstyle{definition}
\newtheorem{defn-tweak}[subsection]{Definition}
\Crefname{defn-tweak}{Definition}{Definitions}
\newtheorem{eg-tweak}[subsection]{Example}
\Crefname{eg-tweak}{Example}{Examples}
\newtheorem*{rems-tweak}{Remarks}
\newtheorem{rem-tweak}[subsection]{Remark}
\Crefname{rem-tweak}{Remark}{Remarks}

\newtheoremstyle{subsection-tweak}
   {11pt}
   {3pt}%
   {}
   {}%
   {\bfseries}
   {}%
   {.5em}
   {\thmnumber{\@{#1}{}\@{#2}.}%
    \thmnote{~{\bfseries#3.}}}

\theoremstyle{subsection-tweak}
\newtheorem{pp}[numberingbase]{}
\newcommand{\bpp}{\begin{pp}}
\newcommand{\epp}{\end{pp}}

\theoremstyle{subsection-tweak}
\newtheorem{pp-tweak}[subsection]{}

\numberwithin{equation}{numberingbase}

\makeatletter
\def\@tocline#1#2#3#4#5#6#7{
    \begingroup %\hyphenpenalty\@M
    \@ifempty{#4}{%
    }{%
    }%

    \parindent\z@ \leftskip#3\relax \advance\leftskip\@tempdima\relax
    #5\hskip-\@tempdima
      \ifcase #1
       \or\or \hskip 2em \or \hskip 1em \else \hskip 3em \fi%
      #6\nobreak\relax
    \dotfill\hbox to\@pnumwidth{\@tocpagenum{#7}}\par
    \nobreak
    \endgroup
  }
 \def\l@section{\@tocline{1}{0pt}{1pc}{}{}}

\renewcommand{\tocsection}[3]{%
  \indentlabel{\@ifnotempty{#2}{\makebox[1.3em][l]{%
    \ignorespaces#1 \bfseries{#2}.\hfill}}}\bfseries{#3}
    \vspace{1.5pt}}

\renewcommand{\tocsubsection}[3]{%
  \indentlabel{\@ifnotempty{#2}{\hspace*{-0.5em}\makebox[2.1em][l]{%
    \ignorespaces#1#2.\hfill}}}#3
    \vspace{1.5pt}}

\makeatother

\makeatletter
\newcommand\appendix@section[1]{%
  \refstepcounter{section}%
  \orig@section*{Appendix \@Alph\c@section. #1}%
}
\let\orig@section\section
\g@addto@macro\appendix{\let\section\appendix@section}
\makeatother

\DeclareMathVersion{normal2}

\begin{document}

\title{WILDLY COMPATIBLE SYSTEMS AND SIX OPERATIONS}

\author{NING GUO}
\address{D\'{e}partement de Math\'{e}matiques, Universit\'{e} Paris-Sud, Orsay Cedex, France }
\email{ning.guo1@u-psud.fr}
%\urladdr{http://math.berkeley.edu/~kestutis/}
\date{\today}
%\subjclass[2010]{Primary 14F22; Secondary 14F20, 14G22, 16K50.}
%\keywords{Brauer group, \'{e}tale cohomology, perfectoid ring, punctured spectrum, purity.}

\begin{abstract} For a scheme $X$ separated and of finite type over an excellent regular scheme $S$, we define wildly compatible systems of constructible sheaves of modules over finite fields on $X$ for certain vector spaces $V$. The main result is that for $\dim S \leq 1$, wildly compatible systems are preserved by Grothendieck's six operations and Verdier's duality. Finally, for a smooth integral scheme $X$ over a finite field, we prove that all $\ell$-adic compatible systems gives wildly compatible systems.
\end{abstract}
\maketitle

%%%%%%%%%%%%%%%%%%%%%%%%%%%%%%%%%%%%%%%%%%%%%%%%%%%%%%%%%%%%%%%%%%%%%%%%%%%%%%%%%%%%%%%%%%
\section{Introduction}

For a scheme $X$ separated and of finite type over an excellent scheme $S$, we consider a family of locally constant \'etale sheaves $\{\CF_i\}_{i\in I}$ of $\mathbb{F}_{\lambda_i}$-modules on $X$. Each sheaf $\CF_i$ corresponds to a representation $\rho_i$ of the \'etale fundamental group of $X$. The wild ramification of $\rho_i$ is described by Swan conductors in terms of Brauer traces (see \cite{Ser77}*{18.1 and 19.2}). To compare wild ramifications of classes $[\CF_i]$ (so-called virtual sheaves) in the Grothendieck group $K_c(X, \mathbb{F}_{\lambda_i})$, it is natural to find a subgroup (defined by Brauer traces) of $\prod_{i\in I}K_c(X,\mathbb{F}_{\lambda_i})$ to specify virtual sheaves with the same ramification. In addition, having ``the same wild ramification'' should be compatible with functorial operations (for instance, Grothendieck's six operations and Verdier's duality). Wild ramifications of sheaves, in turn, have been related to Euler--Poincar\'e characteristics and cycles and the development is summarized as the following.
\benumr
\item Deligne--Illusie (\cite{Illusie}) defined the notion of ``same wild ramification'' for an open subvarieties $X$ of a normal proper variety $Z$ over an algebraic closed field of positive characteristic, and they proved that two locally constant sheaves over $X$ with the same rank and the same wild ramification along $Z\backslash X$ have the same Euler--Poincar\'e characteristic.

\item When $S$ is an excellent Henselian discrete valuation ring or a field with the residue characteristic exponent $p\geq 1$,  Vidal (\cite{VidalI}, \cite{VidalII}) constructed a subgroup of $K_c(X,\mathbb{F}_{\lambda})$ to define the notion ``virtual wild ramification $0$'' and proved its preservation under Grothendieck's six operations.

\item With Vidal's setting, Yatagawa (\cite{Yatagawa}) defined the subgroup of ``the same wild ramification'' and proved that it is preserved by four operations and Verdier's duality. Further, Saito and Yatagawa (\cite{saito2016wild}) proved that two constructible complexes with the same wild ramification have the same characteristic cycle.
\eenum

In this article, we introduce wildly compatible system of (virtual) constructible sheaves as a more general notion, and prove its preservation under Grothendieck's six operations and Verdier's duality. After generalizing Vidal's definition \cite{VidalI}*{2.1} of local wild part of $\pi_1\p{X}$ to excellent base $S$ without restriction on $\dim S$ (\Cref{wild part of pi}), we establish a formula (\Cref{global-local}) connecting global and local wild parts of $\pi_1$ to find their functorial properties (see Proposition \ref{hahaha},  Corollaries \ref{functorial} and \ref{myrealfunctorial}). With these generalities and properties, wildly compatible systems (\Cref{constr-compatible}) have no restriction on $\dim S$ and rely on certain vector spaces $V$, which are flexible: by changing $V$, we obtain Vidal's and Yatagawa's notion as special examples (\Cref{Compare}). A wildly compatible system is defined by a subgroup $K_c(X/S, I, V)\subset \prod_{i\in I}K_c(X, \mathbb{F}_{\lambda_i})$, and sheaves in $K_c(X/S, I, V)$ can be viewed as having ``compatible wild ramifications'': Brauer traces of these sheaves on each stratum are components of some vectors in $V$. The main result is the preservation of wildly compatible systems as the following.
\begin{thm}[\Cref{sevenoperators}]
Let $S$ be an excellent regular scheme with $\dim S\leq 1$. Let $f: X\to Y$ be an $S$-morphism of schemes over $S$ separated and of finite type. We have:
\benumr
\item $f^{\ast}: K_{c}(Y, I)\to K_c(X, I)$ induces $\fust: K_{c}(Y/S, I,V)\to K_{c}(X/S, I,V)$;
\item $\flst\, : K_c(X, I)\to K_c(Y, I)$ induces $\flst\;\! : K_{c}(X/S, I,V)\to K_{c}(Y/S, I,V)$;
\item $\fls\,\; : K_c(X, I)\to K_c(Y, I)$ induces $\fls\ : K_{c}(X/S, I,V)\to K_{c}(Y/S, I,V)$;
\item $\fus\,\, : K_c(X,I)\to K_c(Y, I)$ induces $\fus\; : K_{c}(X/S, I,V)\to K_{c}(Y/S, I,V)$;
\item $\dhom: K_c(X, I)\times K_c(X, I)\to K_c(X, I)$ induces
 \[
 \text{$\dhom: K_{c}(X/S, I,V)\times K_{c}(X/S, I,V)\to K_{c}(X/S, I,V)$ when $V$ is a subalgebra;}
 \]
\item $-\otimes-: K_c(X, I)\times K_c(X, I)\to K_c(X, I)$ induces
\[
\text{$-\otimes-: K_{c}(X/S, I,V)\times K_{c}(X/S, I,V)\to K_{c}(X/S, I,V)$ when $V$ is a subalgebra;}
\]
\item $D_X: K_c(X,I)\to K_c(X,I)$ induces $D_X: K_{c}(X/S, I,V)\to K_{c}(X/S, I,V)$.
\eenum
\end{thm}

This generalizes \cite{VidalII}*{Cor.~0.2} and \cite{Yatagawa}*{Cor.~4.1}. Finally, for a smooth integral scheme $X$ over a finite field, compatible $\ell$-adic virtual constructible sheaves are related to our wildly compatible virtual constructible sheaves. More precisely, the decomposition map $d: K(X, E)\to K(X, F)$ sends compatible $\lambda_i$-adic virtual constructible sheaves to wildly compatible virtual constructible $\mathbb{F}_{\lambda_i}$-sheaves (Theorem \ref{ladic}).

%%%%%%%%%%%%%%%%%%%%%%%%%%%%%%%%%%%%%%%%%%%%%%%%%%%%%%%%%%%%%%%%%%%%%%%%%%%%%%%%%%%%%%%%%%%%%%%%%%%%%%
\subsection{Conventions} By $\p{X/S}$ we always mean a scheme $X$ separated and of finite type over an excellent scheme $S$ (and $Y/R$ is similar). We use $\ov{s}\ra S$ to denote geometric points of $S$, and $X_{\ov{s}}$ is the base change of $X$ over the strict Henselization $S_{\ov{s}}$. 
%%%%%%%%%%%%%%%%%%%%%%%%%%%%%%%%%%%%%%%%%%%%%%%%%%%%%%%%%%%%%%%%%%%%%%%%%%%%%%%%%%%%%%%%%%%%%%%

\subsection*{Acknowledgements.} I thank to my advisor, Weizhe Zheng. With his forward-looking instruction and guidance, this article was written smoothly. He also helped me to correct mistakes and defects with infinite patience. I thank to Lei Fu for important suggestions for revision.

%%%%%%%%%%%%%%%%%%%%%%%%%%%%%%%%%%%%%%%%%%%%%%%%%%%%%%%%%%%%%%%%%%%%%%%%%%%%%%%%%%%%%%%%%%%%%%%%

\section{The wild part of $\pi_1$}
\label{wildpart}
Nagata's theorem predicts that every scheme $X$ separated and of finite type over an excellent scheme $S$ has a compactification $\ov{X}$. It turns out that wild ramifications of constructible \'etale $\mathbb{F}_{\lambda_i}$-sheaves $\CF_i$ over $X$ is considered over the boundary $\ov{X}\backslash X$. By the theory of modular representation, the elements of $\pi_1(X)$ whose super-natural orders are prime to $\lambda_i$ form a set $\pi_1(X)_{\text{reg}}$, where we take Brauer traces of $\CF_i$ (\ref{Brauer-traces} or see \cite{Ser77}*{18.1}). After a recollection of Nagata compactifications (\ref{Ncpt}--\ref{Nndcpt}), we prove a global-local formula (\Cref{global-local}) of wild parts of $\pi_1$ (\Cref{wild part of pi}), which corresponds to $\pi_1(X)_{\text{reg}}$ referred above. By this formula, any wildly compatible system is locally wildly compatible (\Cref{compatiblestrictlocal} (ii)).

\bpp[Nagata compactifications]\label{Ncpt}
For a scheme $X$ over a quasi-compact quasi-separated scheme $S$, we say $f: X\to S$ is \emph{compactifiable} or has a \emph{compactification}, if $f$ is factorized by an $S$-scheme $\ov{X}$ as $f\colon X\stackrel{i}{\hookrightarrow}\ov{X}\stackrel{\ov{f}}{\ra} S$, where $i$ is an open immersion and $\overline{f}$ is a proper morphism. We denote this compactification of $f$ by $\p{\ov{X}, i, \ov{f}}$. By Nagata's compactification theorem (\cite{Conrad}*{Thm.~4.1}), $f$ is compactifiable if and only if it is of finite type and separated. In this case, all compactifications of $f$ form a category $\CC$. The objects of $\CC$ are triplets $(\overline{X}, i, \overline{f})$. A morphism between two objects $(Z_1, i_1, f_1), (Z_2, i_2, f_2)$ in $\CC$ is such a morphism $g: Z_1\to Z_2$ so that $g\circ i_1=i_2$ and $g\circ f_2=f_1$. In this case, we say that $Z_2$ is \emph{dominated} by $Z_1$, or $Z_1$ \emph{dominates} $Z_2$ as compactifications of $X\to S$.

In fact, $\CC$ is a cofiltered category (\cite{Deligne}*{3.2.6}):
\benum
\item[1)] given $Z_1, Z_2\in Obj\p{\CC}$, there exists $Z_3\in Obj\p{\CC}$ and morphisms
    $Z_1\stackrel{g_1}{\la}Z_3\stackrel{g_2}{\ra}Z_2$;
\item[2)] given two morphisms $Z_1 \overset{r}{\underset{s}\rightrightarrows} Z_2 $
in $\CC$, there exists $Z_3\in Obj\p{\CC}$ and a morphism $t: Z_3\to Z_1$ in $\CC$ such that $s\circ t=r\circ t$.
\eenum
   For 1), it suffices to take $Z'\ce Z_1\times_S Z_2$ and a compactification $Z_3$ of $(i_1, i_2): X\to Z'$. For 2), we denote by $k: K\ra Z_1$ the equalizer of $r$ and $s$. By universal property of equalizer, we factorize $i_1$ as $i_1=k\circ i$. Even though $k$ is proper and $K$ is proper over $S$, the morphism $i$ is not necessarily be an open immersion. Therefore, we take a compactification of $i: X\to K$ as $X\stackrel{i_3}{\hookrightarrow} Y \stackrel{p}{\to} K$. The composition $k\circ p: Y\to Z_1$ is the desired morphism.

In the sequel, we assume that $S$ is an excellent scheme.
\epp

\bpp[Nagata normal compactifications] \lab{ngtnormcmp} For $\p{X/S}$ with normal connected $X$ and a Nagata compactification $f\colon X\ra \ov{X}\stackrel{\ov{f}}{\ra}S$, we consider the normalization $\ov{X}^{\nu}$ of $\ov{X}$.
\epp
\bcl
The normalization $\ov{X}^{\nu}$ is also a Nagata compactification of $f: X\to S$.
\ecl
\bpf
By the universal property of normalization, there is a morphism $i: X\to \overline{X}^{\nu}$. We show that $X\stackrel{i}{\longrightarrow}\overline{X}^{\nu}\stackrel{\ov{f}\circ \nu}{\longrightarrow}S$ is another compactification of $f\colon X\to S$. By excellence assumption of schemes over $S$, the normalization $\nu$ is finite so $\ov{f}\circ \nu\colon \ov{X}^{\nu}\to S$ is proper. On the other hand, the base change $Y\ce \overline{X}^{\nu} \times_{\overline{X}} X\ra X$ is a finite birational morphism from an integral scheme to a normal connected scheme, hence is an isomorphism. Therefore, $i\colon X\to \ov{X}^{\nu}$ is an open immersion.
\epf

All normal compactifications of $X\to S$ form a full subcategory $\CC'$ of $\CC$. Further, $\CC'$ is co-cofinal:
\begin{itemize}
\item[1)] for every $x\in Obj(\CC)$, there is an object $y\in Obj(\CC')$ and a morphism $y'\to x$ in $Morph(\CC)$;
\item[2)] given $x\in Obj(\CC)$ and $y_1,y_2\in Obj(\CC')$ with $y_1\to x$ and $y_2\to x$, there exists $y_3\in Obj(\CC')$ fitting into the following commutative diagram
$$
\begin{tikzcd}
y_1 \arrow{dr}
& y_{3} \arrow{l} \arrow[rightarrow]{r} \arrow{d}
& y_{2}\arrow{ld} \\
& x
& \qq .
\end{tikzcd}
$$
\end{itemize}
It is clear that 1) and 2) follows directly from the fact that for any two normal compactifications in $\CC'$, there is always a normal compactification dominating them.

\bpp[Nagata normal dense compactifications] \label{Nndcpt}
 On the basis of \ref{ngtnormcmp}, we consider normal compactifications $\overline{X}$ such that $X\to \overline{X}$ are dominant open immersions. Here we assume that $X$ is a normal connected scheme separated and of finite type over an excellent scheme $S$. In this case, for an arbitrary normal compactification $\overline{X}$, we can always take the connected component $\overline{X}'$ of $\overline{X}$ that contains $X$. Hence $X$ is open dense in $\overline{X}'$, which is also a normal compactification of $X\to S$ dominating $\overline{X}$. Subsequently, all normal compactifications of $X\to S$ such that $X$ is an open subscheme form a cofiltered co-cofinal subcategory of $\CC'$ (and of $\CC$).

By the results about co-cofinality above, for convenience, we always assume that $X$ is an open dense subscheme in every normal compactification of $X\to S$.
\epp

\bd[Wild part of $\pi_1$] \label{wild part of pi}
For an open subscheme $X$ of a normal connected scheme $\ov{X}$, we let $x$ be a geometric point of $\ov{X}$ with the characteristic exponent $p\geq 1$. The strict Henselization $\ov{X}_x$ of $\ov{X}$ at $x$ has a fiber product $X_x\ce \ov{X}_x\times_{\ov{X}}X$. Let $v$ be a geometric point of $X_x$ lying above a geometric generic point $\eta$ of $X$. We define some subsets of $\pi_1(X,\overline{\eta})$ as follows:
\benum
\item[\up{i}] $E_{x,v}$ is the union of images of $p$-Sylow subgroups of $\pi_1(X_x,v)$ in $\pi_1(X,\eta)$;
\item[\up{ii}] $E_{X,x}$ is the union of conjugates of $E_{x,v}$ by $g$ in $\pi_1(X,\eta)$;
\item[\up{iii}] $E_{X,\ov{X}}$ is the closure of the union of $E_{X,x}$ for all geometric points $x$ of $\ov{X}$.
\eenum
For $\p{X/S}$ with normal connected $X$, we define:
\benum
\item[\up{iv}] $E_{X/S}$ is the intersection of $E_{X,\ov{E}}$ for all normal compactifications $\ov{X}$ of $X\to S$;
\item[(v)] for the image $\tau$ of $\eta$ in $X$, the subset $E_{\tau/S}$ is the intersection of images of $E_{U/S}$ in $E_{X/S}$ for all open dense subscheme of $X$.
\eenum
\ed
\bcl \label{normln-cptfn}
\Cref{wild part of pi} (ii) is independent of choices of $v$.
\ecl
\bpf
For another geometric point $w$ above ${\eta}$, there is an isomorphism of fibre functors $\gamma: F_{v}\isom F_{w}$ inducing an isomorphism
\[
\ba
\phi\colon \pi_1(X_x, v)=\Aut(F_v) & \isomto \Aut(F_w)=\pi_1(X_x,w) \qq
h  \longmapsto  \gamma\circ h \circ \gamma^{-1}.
\ea
\]
Therefore, the image $g$ of $\gamma$ in $\pi_1(X,\eta)$ induces an inner automorphism such that the following diagram
\[
\begin{tikzcd}
  \pi_1(X_x, v) \ar{r} \ar{d}{\phi} & \pi_1(X, \eta)  \ar{d}{(-)^g}\\
  \pi_1(X_x, w) \ar{r} & \pi_1(X, \eta)
\end{tikzcd}
\]
commutes, where $(-)^g$ denotes the inner automorphism induced by $g$.
\epf

\brem \label{Vidalcompare}
Here we compare our Definition \ref{wild part of pi} with Vidal's notion (\cite{VidalI}*{\S\S2.1}, \cite{VidalII}*{\S1}).
\begin{itemize}
\item Let $S$ be the spectrum of an excellent Henselian discrete valuation ring of mixed characteristic $(0,p)$ (that is, the characteristic of residue field at generic point is $0$ and for closed point is $p$). Let $\overline{x}$ be the geometric point above $\eta$, then $E_{X,\overline{x}}=\{1\}$. On the other hand, $E^{\text{Vidal}}_{X,{x}}$ is not trivial in general. Therefore, $E_{X,\overline{x}}\neq E^{\text{Vidal}}_{X,{x}}$ in this case.
\item Further, we consider $E_{X,\overline{X}}$ and $E^{\text{Vidal}}_{X,\overline{X}}$ for a normal compactification $\overline{X}$ of $X\to S$. A finite quotient $Q$ of $\pi_1(X)$  corresponds to a  Galois cover $Y\to X$. Let $\overline{Y}$ be the normalization of $\overline{X}$ in $Y$.  From \cite{VidalII}*{Proof~6.1} we know that $g\in Q$ belongs to $E_{X,\overline{X}}$ if and only if $g$ fixes a point of $(\overline{Y})_s$, where $s$ is the closed point of $S$. Similarly, $g\in Q$ belongs to $E^{\text{Vidal}}_{X,\overline{X}}$ if and only if $g$ fixes a point of $\overline{Y}$. But the locus of fixed points $\overline{Y}^g$ in $\overline{Y}$ of $g$ is proper over $S$, since proper morphism is stable under composition. Consequently, we find that $\overline{Y}^g\cap \overline{Y}_s\neq \phi$ and $E_{X,\overline{X}}=E^{\text{Vidal}}_{X,\overline{X}}$.
\end{itemize}
\erem
\brem
With Definition \ref{wild part of pi} we can also define tameness for Galois covers, which is a classical notion. Let $X$ be a  normal connected scheme separated and of finite type over $\Spec (\mathbb{Z})$. A Galois cover $Y\to X$ is called \emph{tame} if $E_{X/\Spec (\mathbb{Z})}(Q)=\{1\}$, where $Q$ is the finite quotient of $\pi_1(X)$ associated to the Galois cover $Y\to X$. When $Y$ is regular, this is the notion of numerical tameness in \cite{KerzSchmidt}.
\erem

\subsection*{A global-local formula}
Given a geometric point $\ov{s}$ of $S$, the strict Henselization $S_{\overline s}\to S$ at $\overline{s}$ induces a base change of $X$ and $\overline{X}$. Though the base change $X_{\ov{s}}$ is not connected in general, we can also define its normal compactification as the disjoint union of normal compactifications of its all connected components. All normal compactifications of $X_{\ov{s}}$ also form a cofiltered category. In this sense, the base change $\p{\ov{X}}_{\ov{s}}$ is a normal compactification of $X_{\overline s}$ over $S_{\overline s}$. Conversely, given a normal compactification $Z_{\ov{s}}$ of $X_{\overline s}\to S_{\overline s}$, we want to prove that it can be dominated by the base change of a normal compactification of $X\to S$ (\Cref{descent}), which leads to a formula relating global and local wild parts of $\pi_1$ as the following.

\begin{thm}\label{global-local}
For $\p{X/S}$ with $\dim S<3$ and normal connected $X$, we have
$$E_{X/S}=\overline{\bigcup_{\overline s\to S}\bigcup_{g\in \pi_1(X,\eta)}g^{-1}E'_{X_{\overline s}/S_{\overline s}}g},$$
where $E'_{X_{\overline s}/S_{\overline s}}$ are images of $E_{X_{\overline s}/S_{\overline s}}$ in $\pi_1(X,\eta)$.
\end{thm}

\begin{proof}
We may assume that $S$ is connected. By Definition \ref{wild part of pi},
\[
\text{
$E_{X, \overline{X} }=\overline{\bigcup_{{x}}\bigcup_{g\in \pi_1(X,\eta)}g^{-1}E_{{x},{v}}g}$, \q where ${x}$ goes through geometric points of $\ov{X}$.
}
\]
 For a geometric point $x$ of $\ov{X}$, its composition with $\ov{X}\to S$ is a geometric point $s$ of $S$. So $x$ gives rise to a geometric point $y$ of $\p{\ov{X}}_{\ov{s}}$. We have the following commutative diagram
\[
\begin{tikzcd}
  \pi_1((X_{\ov{s}})_{y}, w) \ar{r} \ar{d} & \pi_1(X_{\ov{s}}, \xi) \ar{d}  \\
  \pi_1(X_{x}, v) \ar{r} & \pi_1(X, \eta),
\end{tikzcd}
\]
where $\xi$ is a geometric generic point above $\eta$ and $w$ is a geometric point above $\xi$. The first column is an isomorphism, since $(X_{\ov{s}})_{y}\isom X_x$. It follows that the image $E'_{y,w}$ of $E_{y,w}$ in $\pi_1(X,\eta)$ is $E_{x,v}$, and
\begin{align*}
  E_{X,\ov{X}} & = \overline{\bigcup_{\ov{s}}\bigcup_{y}\bigcup_{g\in \pi_1(X,\eta)}g^{-1}E_{ {y},{w}}g} \\
   & =\overline{\bigcup_{\overline{s}}\bigcup_{g\in \pi_1(X,\eta)}g^{-1}\ E'_{X_{\overline{s}}, \p{\ov{X}}_{\ov{s}}}\ g},
\end{align*}
where $\ov{s}$ goes through geometric points of $S$ and $y$ goes through geometric points of $(\ov{X})_{\ov{s}}$. For a fixed $\overline{s}
\to S$, by Proposition \ref{descent} below, the base changes of $\overline{X}$ are co-cofinal in the category of normal compactifications of $X_{\overline{s}}\to S_{\overline{s}}$. By taking intersections on both sides, we find
\begin{align*}
E_{X/S}= \bigcap_{X\to\ov{X}\to S}E_{X, \ov{X} }& \numeq{1} \bigcap_{X\to\ov{X}\to S}\biggl(\ov{\bigcup_{\ov{s}\to S}\bigcup_{g\in \pi_1(X,{\eta})}g^{-1}E'_{X_{\ov{s}}, (\ov{X})_{\ov{s}}}g}\biggr) \\
&\numeq{2} \varprojlim_{Q}\bigcap_{X\to\ov{X}\to S}\biggl(\Bigl(\bigcup_{\ov{s}\to S}\bigcup_{g\in \pi_1(X,{\eta})}g^{-1}E'_{X_{\ov{s}}, (\ov{X})_{\ov{s}}}g\Bigr)(Q)\biggr) \\
& \numeq{3} \varprojlim_{Q}\Bigl(\bigcup_{\ov{s}\to S}\bigcup_{g\in \pi_1(X,{\eta})}g_Q^{-1}\bigcap_{\ov{X}}E'_{X_{\ov{s}}, (\ov{X})_{\ov{s}}}(Q)g_Q\Bigr)\\
& \numeq{4} \varprojlim_{Q}\Bigl(\bigl(\bigcup_{\ov{s}\to S}\bigcup_{g\in \pi_1(X,{\eta})}g^{-1}E'_{X_{\ov{s}}/S_{\ov{s}}}g\bigr)(Q)\Bigr)\\
& \numeq{5} \ov{\bigcup_{\ov s\to S}\bigcup_{g\in \pi_1(X, {\eta})}g^{-1}E'_{X_{\ov s}/S_{\ov s}}g}.
\end{align*}
Here $Q$ goes through all finite quotients of $\pi_1(X, {\eta})$ and $g_{Q}$ denotes the image of $g$ in $Q$. The equalities (2) and (5) come from that a closed subset in a profinite group is the inverse limit of all its finite quotients. The (3) is by the distribution law for intersection and union and the fact that the category of normal compactifications is cofiltered.
\end{proof}

\bprop\label{descent}
We let $X$ and $S$ be as assumed in \Cref{global-local}. For any normal compactification $Z_{\ov{s}}$ of $X_{\ov{s}}\to S_{\ov{s}}$, there exists a normal compactification $\ov{X}$ of $X\to S$ such that $\ov{X}\times_S S_{\ov{s}}$ dominates $Z_{\ov{s}}$.

%Let $S$ and $X$ be the same as in Theorem \ref{global-local}.  Let $\ov{s}$ be a geometric point of $S$ and let $S_{(\overline{s})}\to S$ be the strict Henselization of $S$ at $\overline{s}$. Let $X_{(\overline{s})}$ be the base change of $X$ to $S_{(\overline{s})}$. Then given an arbitrary normal compactification $\overline{X_{(\overline{s})}}$ of $X_{(\overline{s})}\to S_{(\overline{s})}$, there exists a normal compactification $\overline{X}$ of $X\to S$ such that its base change by $S_{(\overline{s})}\to S$ dominates $\overline{X_{(\overline{s})}}$ as a normal compactification of $X_{(\overline{s})}\to S_{(\overline{s})}$.
\eprop

\bpf
The base change of any normal compactification of $X\to S$ to $S_{\overline{s}}$ is a normal compactification of $X_{\overline{s}}\to S_{\overline{s}}$. To find the desired $\ov{X}$, we prove in the following steps.
\benumr
\item We have $S_{\ov{s}}=\varprojlim_{i\in I}S_i$ for each affine scheme $S_i$ \'etale over $S$. Because $X_{\ov{s}}$ and $Z_{\ov{s}}$ are of finite type over $S_{\ov{s}}$ and $S_{\ov{s}}$ is excellent, a fortiori Noetherian, by \cite[\href{https://stacks.math.columbia.edu/tag/01TX}{01TX}]{SP}, $X_{\ov{s}}$ and $Z_{\ov{s}}$ are of finite presentation over $S_{\ov{s}}$. By \cite{EGAIV3}*{8.8.2}, there exists $i_0\in I$ such that the normal compactification $X_{\ov{s}}\ra Z_{\ov{s}}$ descend over $S\pr\ce S_{i_0}$ as $X_{S\pr}\ra Z_{S\pr}$.
\item Since $S'\to S$ is \'etale and affine, it is quasi-finite and separated. By Zariski's main theorem, it is an open immersion followed by a finite morphism as $S'\to T\to S$. We may assume that $T$ is connected and replace $T$ by its normalization (as in the proof of \Cref{normln-cptfn}) if necessary.
\item We glue $X_{T}$ and $Z_{S\pr}$ along the open $X_{S\pr}$ to obtain $Z\ce Z_{S\pr}\cup_{X_{S\pr}}X_T$. It is clear that $Z\ra T$ is of finite type and we prove its separatedness. We restrict the diagonal morphism $\Delta_{Z/T}\colon Z\ra Z\times_{T}Z$ to four open subschemes $Z\times_{T}Z=(Z_{S'}\times_{T}Z_{S'})\cup (Z_{S'}\times_{T}X_T)\cup (X_T\times_{T}Z_{S'})\cup (X_T\times_{T}X_T)$ to check $\Delta_{Z/T}$ is a closed immersion. Since $Z_{S\pr}/T$ and $X_T/T$ are separated, by symmetry, it suffices to consider $\Delta_{Z/T}|_{Z_{S\pr}\times_{T}X_T}$, which is the closed immersion $\Delta_{X_{S\pr}/S\pr}$ by the isomorphism $Z_{S\pr}\times_T X_T\isom Z_{S\pr}\times_{S\pr}X_{S\pr}$ and the separatedness of $X_{S\pr}/S\pr$. Therefore, by Nagata's theorem, there is a compactification $Z_T$ of $Z\ra T$. Further, we have $\p{Z_T}_{S\pr}\isom Z_{S\pr}$. 
\item Since $T\to S$ is finite, it follows that $X_T\to Z_T\to S$ is also a compactification. Because $T$ is normal and $\dim T<3$, it is Cohen-Macaulay. Subsequently, the dominant morphism $T\to S$ is flat and so is $X_T\to X$. By \Cref{finitedominate} below, there is a compactification $Z_T\pr$ of $X_T\to S$ and a normal compactification $\ov{X}$ of $X\ra S$ with a finite morphism $Z_T\pr \ra \ov{X}$. The base change of $Z_T\pr\ra \ov{X}$ to $S_{\ov{s}}$ is a finite morphism between two normal compactifications of $X_{\ov{s}}\ra S_{\ov{s}}$, so it is an isomorphism and $\p{\ov{X}}_{\ov{s}}$ dominates $Z_{\ov{s}}$. \qedhere
\eenum 
\epf

\blem\label{finitedominate}
For a finite flat morphism $f: X\to Y$ between two connected schemes separated and of finite type over $S$ and a compactification $\overline{X}$ of $X\to S$, there is a compactification $\overline{X}'$ of $X\to S$ dominating $\overline{X}$ and a  compactification $\overline{Y}'$ of $Y\to S$ fitting into a commutative diagram
\[
\begin{tikzcd}
    X \arrow{r} \arrow{d}[swap]{f} & \overline{X}' \arrow{d}{\overline{f}} \\
    Y \arrow{r}           & \overline{Y}'
  \end{tikzcd},\q \text{where $\overline{f}$ is finite flat.}
\]

\elem

\begin{proof}
This proof comes from O. Gabber's idea and is similar to I. Vidal's proof in \cite{VidalI}*{Prop.~2.1.1~(ii)}.

By Nagata's theorem, there exists $\overline{Y}$ as a compactification of $Y$. The closure of the graph of $f$ in $\overline{X}\times_S \overline{Y}$ is denoted by $\Gamma$.  By the separatedness of $\overline{Y}\to S$, the first projection $p_1: \Gamma\to \overline{X}$ is an isomorphism $i: \Gamma_X\isom X$ when restricted to $X$.  Because $f$ is dominant, the second projection $p_2: \Gamma\to \overline{Y}$ when restricted to $Y$ coincides with $f$, up to a composition with the isomorphism $i$. Therefore, the morphism $p_2$ is finite and flat and satisfies the condition in the flattening theorem (\cite{RG}*{Thm.~5.2.2})  for $n=0$: it is flat outside $\overline{Y}\backslash Y$. Then there exists a $Y$-admissible (that is, the blowing-up locus is in $\overline{Y}\backslash Y$) blowing-up $\overline{Y}'\to \overline{Y}$ such that the strict transform $\overline{f}: \overline{X}'\to \overline{Y}'$ is flat. The flatness of $\overline{f}$ implies it is not only finite on $Y\subset \overline{Y}'$ but also on the entire $\overline{Y}'$.
\end{proof}

\brem\label{ChenyangXu}
If we consider normal compactifications which are algebraic spaces, the assumption in Theorem \ref{global-local} for $S$ can be more general: $S$ can be an arbitrary  excellent scheme. The idea is from Chenyang Xu, who suggests to consider equivariant normal compactifications. As in the proof of Theorem \ref{global-local} (ii) let $K(S')$ and $K(S)$ be fraction fields of $S'$ and $S$ respectively. Then we take a Galois extension $K/K(S)$ with Galois group $G$ containing $K(S')$ and the normalization $R$ of $S$ in $K$. The normalization of $S$ in $K(S')$ is denoted by $T$. With the similar technique in (iii) before, one can extend $\overline{X_{S'}}$ to a normal compactification $\overline{X_T}$ of $X_T\to T$ ($X_T$ can be replaced with its normalization if necessary). Then the base change of $\overline{X_T}$ to $R$ can be dominated by a $G$-equivariant normal compactification of $X_R\to R$ by the construction in \cite{Zhengindependence}*{3.7}. The normalization of the quotient $\overline{X_R}//G$ is a normal compactification and its base change along $S'\to S$ dominates $\overline{X_{S'}}$.
\erem

\subsection{Funtoriality of wild parts of $\pi_1$}

The wild parts of $\pi_1$ have pleasant functorial properties. The following lemma generalizes Vidal's results (\cite{VidalI}*{Prop.~2.1.1}).

\bprop\label{hahaha} In this proposition, we fix
\begin{itemize}
\item $S$ and $R$ are excellent scheme;
\item $X\to \overline{X}$ and $Y\to \overline{Y}$ be two open immersions of normal connected schemes;
\item $f: X\to Y$ is a morphism of schemes;
\item $\phi: \pi_1(X, \eta)\to \pi_1(Y, \xi)$ is the morphism of \'etale fundamental groups induced by $f$, where $\eta$ is a geometric generic point of $X$ and $\xi$ is the image of $\eta$ in $Y$.
\end{itemize}
\benumr
\item
\[
\text{If there is a commutative square}\q
\begin{tikzcd}
\overline{X} \arrow{r}
& \overline{Y}\\
X \arrow[hookrightarrow]{u}\arrow{r}{f}
& Y\arrow[hookrightarrow]{u},
\end{tikzcd}
\q \text{then} \q \phi(E_{X, \overline{X}})\subset E_{Y,\overline{Y}}.
\]
\item Let $X\ra \ov{X}\ra S$ and $Y\ra \ov{Y}\ra R$ be two compactifications.
\[
\text{If there is a commutative square}\q
\begin{tikzcd}
X\arrow{r}{f}\arrow{d}
& Y\arrow{d}\\
S\arrow{r}
&R,
\end{tikzcd}
\q \text{then} \q \phi(E_{X/S})\subset E_{Y/R}.
\]
\item When $f$ is a finite \'etale, we have $E_{X, \overline{X}}=E_{Y, \overline{Y}}\cap \pi_1(X, \eta)$.
If further $\p{X/S}$ and $\p{Y/S}$, then $$E_{X/S}=E_{Y/S}\cap \pi_1(X,\eta).$$
\item With the condition of (ii), if we further assume that $X=Y$ with $f=\id$, then $$E_{X/S}\subset E_{X/R}.$$ When $S\to R$ is proper, we have $$E_{X/S}=E_{X/R}.$$
\eenum
\eprop
\bpf \hfill
\benumr
\item By definition, $$E_{X, \ov{X}}=\ov{\bigcup_{x\to \ov{X}}\bigcup_{g\in \pi_1(X, \eta)}g^{-1}E_{x,v}g},$$ where $x$ and $v$ are geometric points of $\ov{X}$ and $X_x=X\times_{\ov{X}}\ov{X}_{{x}}$ respectively, and $v$ is above $\eta$.  It suffices to show that
    \[
    \phi\p{\bigcup_{x\to \ov{X}}\bigcup_{g\in \pi_1(X, \eta)}g^{-1}E_{x,v}g}\subset E_{Y, \ov{Y}},
    \]
    since $E_{Y,\ov{Y}}$ is closed. Any element $\sigma\in \bigcup_{x\to \ov{X}}\bigcup_{g\in \pi_1(X, \eta)}g^{-1}E_{x,v}g$ can be written as $\gamma_1^{-1}\phi_1\gamma_1$, where $\phi_1\in \pi_1(X_x, v)$ and $\gamma_1\in \pi_1(X, \eta)$. We let $y$ and $w$ be the images of $x$ and $v$ respectively, and consider the following diagram
$$\begin{tikzcd}[column sep=small]
v\arrow{r}
& X_x\arrow{rr}\arrow{dd}\arrow[dashed]{dr}
&
& \ov{X}_{x}\arrow{dd}\arrow{dr}
& \\
&
& Y_y\arrow[crossing over]{rr}
&
& \ov{Y}_{y}\arrow{dd}\\
\eta \arrow{r}
& X\arrow{dr}
&
& \ov{X}\arrow{dr}\arrow[leftarrow]{ll}
&\\
&
& Y\arrow{rr}\arrow[leftarrow, crossing over]{uu}
&
& \ov{Y}\end{tikzcd}$$
and a similar diagram for fundamental groups. These diagrams are not necessarily commutative.
Let $\phi_2$ be the image of $\phi_1$ in $\pi_1(Y_y, w)$ and let $\gamma_2$ be the image of $\gamma_1$ in $\pi_1(Y, \xi)$. Then we have $f(\sigma) =\gamma_2^{-1}\phi_2\gamma_2\in E_{Y,\overline
{Y}}$ and $f(E_{X,\ov{X}})\subset E_{Y, \ov{Y}}$.
\item For any normal compactification $\ov{Y}$ of $Y\to R$, we take a normal compactification of $\ov{Y}\times_R S$ to obtain a normal compactification of $X\to S$. By the result of (i) the assertion follows.
\item
We first prove for the case when $\ov{X}$ is dominated by the normalization of a normal compactification $\ov{Y}$ for $Y\to S$. Then $X_x$ is a pointed connected component of $Y_y\times_{Y} X$. To prove that $E_{x,v}=E_{y,w}\cap \pi_1(X, \eta)$, it suffices to show that the following diagram
\[
\begin{tikzcd}
  \pi_1(X_x,v) \arrow[hookrightarrow]{d} \arrow{r}{\psi_X} & \pi_1(X,\eta) \arrow[hookrightarrow]{d} \\
  \pi_1(Y_y,w) \arrow{r}{\psi_Y} & \pi_1(Y,\xi)
\end{tikzcd}
\]
is cartesian.
All vertical arrows are injective. Let $U=\psi_Y^{-1}(\pi_1(X, \eta))\subset \pi_1(Y_y, w)$. This is an open subgroup associated to an intermediate pointed finite\'etale covering $X_x\to V\to Y_{y}$. The correspondent point is denoted by $t\to V$. As the map $\pi_1(V, t)\to \pi_1(Y, \xi)$ is factorized by $\pi_1(X,\eta)$, then by \cite{SGA1new}*{Exp. V, 6.4}, the natural projection $V\times_{Y} X\to V$ admits a pointed section $s$. Composite the section $s$ with the natural pointed morphism $V\times_{Y} X\to Y_y\times_{Y}X$, we find that the pointed morphism $V\to Y_y\times_{Y}X$ must factors through $X_{x}$. That means $V\isom X_{x}$, hence that diagram is cartesian.

Let $q\in \pi_1(Y, \xi)$ and let $(X',\eta')$ be a pointed connected finite \'etale cover of $Y$ corresponding to the open subgroup $\pi_1(X',\eta')=q\pi_1(X, a)q^{-1}$. We denote by $\phi: \pi_1(X, \eta)\to \pi_1(X', \eta')$ the isomorphism induced by conjugation by $q$ in $\pi_1(Y, b)$ and $\mu: \pi_1(X', \eta')\to \pi_1(X, \eta)$ is denoted for its inverse. Then $q^{-1}E_{y,w}q\cap \pi_1(X, \eta)=\mu(E_{y, w}\cap \pi_1(X',\eta'))=\mu(E_{x,\eta'})$ (for $X'$ it is the same with the above argument). By the result of (i) we find that $\psi(E_{x, v'})\subset \psi(E_{X', \ov{X'}})\subset E_{X, \ov{X}}$. Thus $q^{-1}E_{y,w}q\cap \pi_1(X, \eta)\subset E_{X, \ov{X}}$. As this can be applied to all $q$ and all geometric point $x\to \ov{X}$, we then use the result of (i) to find that $E_{Y, \ov{Y}}\cap \pi_1(X, \eta)\subset E_{X, \ov{X}}$.
By Lemma \ref{finitedominate} above, we finish the proof.
\item[\up{iv}]  Every normal compactification of $X\to S$ is also a normal compactification of $X\to R$. Further, given any normal compactification $W$ of $X\to R$, there is a normal compactification $Z$ of $X\to R$ dominating $W$. Hence $Z$ is also a normal compactification of $X\to S$ and dominates $W$. Therefore $E_{X/S}=E_{X/R}$. \qedhere
\eenum
\epf
\bdt \label{catfun} \hfill
\benumr
\item Let $\CB$ be the (pointed) category whose objects consist of all open immersions $(X\stackrel{i}{\hookrightarrow}\ov{X})$ of normal connected schemes with fixed geometric generic points. Morphisms in $\CB$ consists of all commutative squares as in Proposition \ref{hahaha} (i) such that induced morphisms of geometric generic points of $Y, \ov{Y}, X, \overline
    {X}$ are compatible, where $(X\hookrightarrow\ov{X})$ and $(Y{\hookrightarrow} \ov{Y})$ are objects in $\CB$.
\item Let $\cD$ be the (pointed) category whose objects consist of all  $(X/S)$, where $X$ are normal connected with fixed geometric generic points. Morphisms in $\cD$ consists of all commutative squares as in Proposition \ref{hahaha} (ii) such that induced morphisms of geometric generic points of $X, S, Y, R$ are compatible, where $(X/S)$ and $(Y/R)$ are objects in $\cD$.
\eenum
\edt
With Definition \ref{catfun} and Proposition \ref{hahaha} (i)\& (ii) we obtain the following corollary.
\begin{cor}\label{functorial}
Here we use the same notations as Definition \ref{catfun}.
\benumr
\item The assignment
\[
  b\colon \CB  \ra \mathbf{Set} \qq
  (X\ra \ov{X}) \mapsto  E_{X,\ov{X}}
\]
is a functor. For each morphism $\phi$ in $\CB$ as in Definition \ref{catfun} (i), $b$ sends it to $b(\phi): E_{X,\ov{X}}\to E_{Y,\ov{Y}}$ defined by $f$.
\item There is a functor $E_{-/-}: \cD\to {\mathbf{Set}}$ defined as follows. For each $(X/S)\in Obj(\cD)$, $E_{-/-}$ sends it to $E_{X/S}$. For each morphism $\psi$ in $\cD$ as in Definition \ref{catfun} (ii), $E_{-/-}$ sends it to $E_{-/-}(\psi): E_{X/S}\to E_{Y/R}$ defined by $f$.
\item Let $(Y/R)$ be an object in $\cD$. Let $f: X\to Y$ be a morphism separated and of finite type. Then $f$ gives a morphism in $\cD$ and hence we have $E_{X/R}\subset E_{Y/R}$.
\eenum
\end{cor}
\begin{cor}\label{myrealfunctorial}
Let $S$ be an excellent scheme. Let $\cD_S$ be the category of all normal connected schemes separated and of finite type over $S$. Let $X$ be a normal connected scheme. Let $\CE^X$ be the category consisting of all excellent schemes $S$ with $X\to S$ separated and of finite type.
\benumr
\item There is a covariant functor $E_{-/S}: \cD_S\to {\mathbf{Set}}$ defined as follows. For each object $X\in \cD_{S}$, $E_{-/S}$ sends it to $E_{-/S}(X)=E_{X/S}$ an object in ${\mathbf{Set}}$. For each morphism $f: X\to Y$ in $\cD_S$, $E_{-/S}$ sends it to $E_{-/S}(f): E_{X/S}\stackrel{f_{\ast}}{\to} E_{Y/S}$ a morphism in ${\mathbf{Set}}$.
\item There is a covariant functor $E_{X/-}: \CE^X\to {\mathbf{Set}}$ defined as follows. For each object $S\in \CE^X$, $E_{X/-}$ sends it to $E_{X/-}(S)= E_{X/S}$. For each morphism $h: S\to R$ in $\CE^X$, $E_{X/-}$ sends it to $E_{X/-}(h): E_{X/S}\hookrightarrow E_{X/R}$ a inclusion morphism in ${\mathbf{Set}}$.
\eenum
\end{cor}
\begin{proof}
\vskip 0.3cm
By Corollary \ref{functorial} (ii) the assertion of (i)\&(ii) follows. \qedhere
\end{proof}

\section{Wildly compatible systems of virtual constructible sheaves}
\label{comp}
Now we consider a family of virtual constructible  (resp., locally constant) sheaves $\{\CF_{i}\}_{i\in I}$ on $X$ in $\p{X/S}$.  Here $\CF_{i}$ are elements in Grothendieck group $K_{c}(X, \mathbb{F}_{\lambda_i})$ (resp., $K_{\text{coh}}(X,\mathbb{F}_{\lambda_i})$) of constructible (resp., locally constant) sheaves of $\mathbb F_{\lambda_i}$-modules on $X$, and each $\lambda_i$ is a power of prime $\ell_i$.  In order to compare the wild ramifications of $\{\CF_{i}\}_{i\in I}$, we use Brauer traces.

\bpp[Brauer traces]\label{Brauer-traces}
Let $G$ be a profinite group. For a finite field $\mathbb{F}_{\lambda}$, where $\lambda$ is a power of a prime $\ell$, the subset consisting of $\ell$-regular elements (i.e., elements of orders prime to $\ell$) is denoted by $G_{\text{$\ell$-reg}}$. For $M\in K_{\cdot}(\mathbb{F}_{\lambda}[G])$ an element of Grothendieck group of finite dimensional $\mathbb{F}_{\lambda}$-vector spaces with continuous $G$-actions, the Brauer trace is a central function $\Tr^{\text{Br}}_M: G_{\text{$\ell$-reg}}\to W(\mathbb{F}_{\lambda})$, where $W(\mathbb{F}_{\lambda})$ is the Witt ring of $\mathbb{F}_{\lambda}$. Concretely, each eigenvalue $\zeta$ of action of $g\in G_{\text{$\ell$-reg}}$ on $M$ is in $\ov{\mathbb{F}}_{\lambda}$ and has a unique lift $[\zeta]$ as a root of unity of order prime to $\ell$, and Brauer trace is given by $\Tr^{\text{Br}}_M(g)=\sum [\zeta]$. Brauer trace is an additive function and is multiplicative with respect to the element of $K_{\cdot}(\mathbb{F}_{\lambda}[G])$.

Note that the Brauer trace, when restricted to $E_{X/S}$, takes value in a CM field $K_{\lambda_i}\ce\mathbb{Q}_{\lambda_i}\cap \mathbb{Q}(\zeta_{p^{\infty}})$ considered as a subfield of $\mathbb{Q}_{\lambda_i}$, where $\mathbb{Q}_{\lambda_i}$ are quotient fields of $W(\mathbb{F}_{\lambda_i})$(i.e., unramified extensions of $\mathbb{Q}_{\ell_i}$ with residue fields $\mathbb{F}_{\lambda_i}$). Let $V$ be a $\mathbb{Q}$-vector subspace of $\prod_{i\in I}K_{\lambda_i}$. In the sequel we use this notation of $K_{\lambda_i}$ for convenience and $\lambda_i$ are all assumed to be invertible on $S$.
\epp

\subsection{Wildly compatible locally constant sheaves}
\begin{defn}\label{coh-compatible}
Let $X$ be a normal connected in $\p{X/S}$. Let $V$ be a fixed subspace of $\prod_{i\in I}K_{\lambda_i}$. We say a system $\{\CF_i\}_{i\in I}$ of virtual locally constant $\mathbb{F}_{\lambda_i}$-sheaves on $X$ is a \emph{wildly compatible system} for $V$ if $(\Tr^{\text{Br}}_{\CF_i}(g))_{i\in I}\in V$ for all $g\in E_{X/S}$. We define a subgroup $K_{\text{coh}}(X/S,I,V)$ of $K_{\text{coh}}(X, I)\ce\prod_{i\in I}K_{\text{coh}}(X, \mathbb{F}_{\lambda_i})$ to be the subgroup consisting of $(\CF_{i})_{i\in I}$ in wildly compatible systems.
\end{defn}

\brem[\cite{VidalI}*{Rem.~2.2.1}]
Let $S$ and $X$ be the same as in Definition \ref{coh-compatible}. Assume that $\{\CF_{i}\}_{i\in I}$ are finitely many locally constant virtual sheaves over $\mathbb{F}_{\lambda_i}$. Let $V$ be a subspace of $\prod_{i\in I}K_{\lambda_i}$. Then $\{\CF_{i}\}$ form a wildly compatible system for $V$ if and only if there exists a normal compactification $\ov{X}_0$ of $X\to S$ such that $\{\CF_{i}\}$ satisfy the following condition: for every $g\in E_{X/S,\ov{X}_0}$, we have $(\Tr^{\text{Br}}_{\CF_i}(g))_{i\in I}\in V$. One side is easy to prove by $E_{X/S}\subset E_{X, \ov{X}_0}$. For the other side, we notice that $\pi_1(X)$ acts on $\CF_{i}$ by a finite quotient $Q$. Since $Q$ is finite and the category $\CC$ of normal compactifications of $X\to S$ is cofiltered, for the projection $p_Q: \pi_1(X,\ov{\eta})\to Q$, the finite quotient parts $E_{X/S}(Q)\ce p_Q(E_{X/S})$ and $E_{X, \ov{X}}(Q)\ce p_Q(E_{X,\ov{X}})$ satisfy $E_{X/S}(Q)=\bigcap_{\ov{X}\in \CC} E_{X, \ov{X}}(Q)=E_{X, \ov{X}_0}(Q)$ for a normal compactification $\ov{X}_0$ of $X\to S$ then the assertion follows.
\erem
\subsection{Wildly compatible constructible sheaves}
Given a virtual constructible sheaf $\CF_{\ell}$ on $X$, there is a stratification $X=\coprod_{i\in I} U_i$ denoted by $P_{\ell}$, where $U_i$ are normal connected and locally closed subschemes of $X$ such that restrictions of $\CF_{\ell}$ on $U_i$ are locally constant. Here by ``stratification'' we mean a ``nice stratification'', i.e., for
each stratum $U_i$, its closure is a disjoint union of strata $\ov{U_i}=\coprod_{j\in J}U_{j}$ where $J\subset I$.  But $X$ can have another stratification $Q_{\ell}$ for $\CF_{\ell}$: $X=\coprod_{j\in J}V_j$. We say a stratification $P_{\ell}$ is \emph{finer} than $Q_{\ell}$ or $P_{\ell}\succeq Q_{\ell}$, if $V_j=\coprod_{i(j)} U_{i(j)}$ for each $j$, here $i(j)$ are functions of index sets $I$ and $J$. Now we consider finitely many virtual constructible sheaves $\CF_{i}$ on $X$ with stratifications $P_{i}$. Then there is a common refinement $P$ so that we can compare their wild ramifications in terms of Brauer traces.
\begin{assumption}\label{VVV}
In the sequel when considering virtual constructible sheaves, we further assume that $V=\bigcap_{J}p_J^{-1}p_J(V)$, where $J$ goes through all finite subsets of $I$ and $p_J$ are projections to subspaces.

\end{assumption}
\begin{defn}\label{constr-compatible} Let $V$ be a fixed $\mathbb Q$-vector subspace of $\prod_{i\in I}K_{\lambda_i}$. We say a system $\{\CF_{i}\}_{i\in I}$ of virtual constructible $\mathbb F_{\lambda_i}$-sheaves on $X$ in $\p{X/S}$ is a \emph{wildly compatible system} for $V$, if for any finite subset $J\subset I$, there is a common stratification $X=\coprod_{\sigma\in \Sigma_J}X_{\sigma}$ such that $\{\CF_{j}|_{X_{\sigma}}\}_{j\in J}$ is a wildly compatible system of virtual locally constant sheaves on $X_{\sigma}$ for $p_J(V)$ and for each $\sigma\in \Sigma_J$. We define a subgroup $K_{c}(X/S, I,V)$ of $K_c(X, I)\ce\prod_{i\in I}K_c(X, \mathbb F_{\lambda_i})$ to be the subgroup consisting of such $(\CF_{i})_{i\in I}$.
\end{defn}
\brem\label{finiteset}
From Definition \ref{constr-compatible}, it is clear that a system $\{\CF_i\}_{i\in I}$ of virtual constructible sheaves is a wildly compatible system for $V$ if and only if its restriction to $J$ is a wildly compatible system for $p_J(V)$ for every finite subset $J\subset I$.
\erem
\begin{prop}\label{stratificationindependence}
For $\p{X/S}$ with $\dim S\leq 1$, let $\{\CF_{i}\}_{i\in I}$ be finitely many virtual locally constant sheaves on $X$. We fix a $\mathbb Q$-vector subspace $V$ of $\prod_{i\in I}K_{\lambda_i}$. If there is a common stratification $X=\coprod_{\sigma\in \Sigma}X_{\sigma}$ for $\{\CF_{i}\}_{i\in I}$ such that for each $\sigma$ we have $(\Tr^{\text{Br}}_{\CF_{i}|_{X_{\sigma}}}(g))_{i\in I}\in V$ for all $g\in E_{X_{\sigma}/S}$,  then $(\Tr^{\text{Br}}_{\CF_{i}}(g))_{i\in I}\in V$ for all $g\in E_{X/S}$.
\end{prop}
\begin{proof}
This propostion and its proof is similar to Vidal's in \cite{VidalII}*{6.2~(ii)}. By the valuative criterion of O. Gabber (\cite{VidalII}*{6.1}), $E_{X/S}=\ov{\bigcup_{\sigma} \mathrm{Im}\p{E_{X_{\sigma}/S}}}$. Hence for every $g\in E_{X/S}$, we have $(\Tr^{\text{Br}}_{\CF_{i}}(g))_{i\in I}\in V$.
\end{proof}
\begin{cor}\label{coh-ccompatible}
Assumption of $X$ being as in Proposition \ref{stratificationindependence},  for any index set $I$ and any $V$, we have $$K_{\mathrm{coh}}(X/S,I,V)=K_{c}(X/S, I,V)\cap K_{\mathrm{coh}}(X, I).$$
\end{cor}

\brem\label{Compare}
Here we give three special examples of wildly compatible systems of virtual locally constant sheaves and constructible sheaves.
\begin{itemize}
\item[\up{i}] When $\#I=1$ and $V=0$, a virtual constructible sheaf $\CF$ in $K_{c}(X/S, I, V)$ can be viewed as it ``has zero wild ramification''.
\item[\up{ii}] When $\#I=2$ and let $V$ be the kernel of the map
$$K_{\lambda_1}\times K_{\lambda_2}\to \mathbb{Q}, \qquad a\mapsto \frac{1}{[K_{\lambda_1}: \mathbb{Q}]}\Tr_{K_{\lambda_1}/\mathbb{Q}}(a)-\frac{1}{[K_{\lambda_2}: \mathbb{Q}]}\Tr_{K_{\lambda_2}/\mathbb{Q}}(a)$$
For $(\CF_1,\CF_2)\in K_{\text{coh}}(X/S, I, V)$ they ``have the same wild ramification'' in a weaker sense.
\item[\up{iii}] Let $L$ be a field of characteristic zero with embeddings $\iota=(\iota_i)_{i\in I}: L\to \prod_{i\in I}\ov{\mathbb Q_{\lambda_i}}$ and let $V$ be $(\prod_{i\in I}K_{\lambda_i})\cap Image((\iota_i)_{i\in I})$. If $(\Tr^{\text{Br}}_{\CF_i}(g))_{i\in I}\in V$ for all $g\in E_{X/S}$, then $(\CF_{i})_{i\in I}\in K_{\text{coh}}(X/S, I,V)$.  In this case,  for any fixed index set $I$, $(\CF_i)_{i\in I}\in K_{\text{coh}}(X/S, I,V)$ if and only if for all subset $J\subset I$ with $\#J\leq 2$, $(\CF_{j})_{j\in J}$ is a wildly compatible system for $p_J(V)$.
\end{itemize}
\begin{itemize}
\item[(a)] For (i), when $S$ is the spectrum of an excellent Henselian discrete valuation ring with residue characteristic exponent $p\geq 1$ or an algebraically closed field $k$ with characteristic $p$, and $X$ is a normal connected scheme separated and of finite type over $S$, by \Cref{Vidalcompare} we obtain $K_{c}(X/S, I, V)=K_{c}(X, \mathbb{F}_{\lambda})^0_t$ which is Vidal's notion.
\item[(b)] For (ii), when $S$ is the spectrum of an excellent Henselian discrete valuation ring of residue characteristic exponent $p\geq 1$ and $X$ is a scheme separated and of finite type over $S$, we obtain $K_{\text{coh}}(X/S, I, V)=\Delta_{coh}(X, \mathbb{F}_{\lambda_1}, \mathbb{F}_{\lambda_2})$ which is Yatagawa's notion. Further with Assumption \ref{VVV} for $V$, we have  $K_c(X/S, I, V)=\Delta_c(X, \mathbb{F}_{\lambda_1},\mathbb{F}_{\lambda_2})$. In fact, our generalization is the same with Kato's (\cite{Kato2016wildram}*{Def.~2.2}).
\item[(c)]Let $I, V$ and $S, X$ be the same with (iii) with extra Assumption \ref{VVV} for $V$. If for any $J\subset I$ a finite subset , there is a stratification $X=\coprod_{\sigma\in \Sigma_J}X_{\sigma}$ such that for each $g\in E_{X_{\sigma}/S}$ there exists $a\in L$ and $\Tr^{\text{Br}}_{\CF_i|X_{\sigma}}(g)=\iota_i(a)$, then $(\CF_i)_{i\in I}\in K_c(X/S, I, V)$. In this case, for any fixed index set $I$, we have $(\CF_i)_{i\in I}\in K_c(X/S, I, V)$ if and only if for all subset $J\subset I$ with $\#J\leq 2$, $(\CF_j)_{j\in J}$ is a wildly compatible system for $V$.
\end{itemize}
\erem

\bprop\label{compatiblestrictlocal}
For $\p{X/S}$, let $T\stackrel{h}{\ra} S\ra R$ be morphisms of excellent schemes with $\p{S/R}$.
\benumr
\item 
If we also denote $X_T\to X$ by $h$, then 
\[
h^{\ast}K_c(X/S,I,V)\subset K_c(X_T/T,I,V).
\]
\item Suppose $S$ is regular with $\dim S<3$. If $S_{\ov{s}}\ra S$ is denoted by $s$, then we have 
\[
\text{
$\p{\CF_i}_{i\in I}\in K_c(X/S,I,V)$\q if and only if \q 
$\p{s^{\ast}\CF_i}_{i\in I}\in K_c(X_{\ov{s}}/S_{\ov{s}},I,V)$\q for all $\ov{s}$.
}
\]

\item  We have $K_c(X/R,I,V)\subset K_c(X/S,I,V)$.
\eenum
\eprop
\bpf \hfill
\benumr
\item It suffices to prove for the case when $\CF_{i}$ are locally constant. The assertion follows directly from Corollary \ref{hahaha},(ii).
\item For $(\CF_{i})_{i\in I}\in K_c(X/S, I, V)$, by (i), we conclude that $\{s^{\ast}(\CF_{i})\}_{i\in I}$ is a wildly compatible system.

For the other side, we reduce to the case when $\#I$ is finite and $\CF_{i}$ are locally constant. There exists a finite quotient $E_{X/S}(Q)$ such that for all $g\in E_{X/S}$, $\Tr^{\text{Br}}_{\CF_i}(g)=\Tr^{\text{Br}}_{\CF_i}(\ov{g})$ for all $i\in I$, where $\ov{g}$ is the image of $g$ in $E_{X/S}(Q)$. By Theorem \ref{global-local}, $E_{X/S}(Q)=\bigcup_{\ov{s}\to S}\bigcup_{h\in \pi_1(X, a)}\ov{h}E'_{X_{(\ov{s})}/S_{(\ov{s})}}(Q)\ov{h^{-1}}$. Thus there is a geometric point $\ov{s_g}\to S$ such that $\ov{g}=\ov{h}\ov{g'}\ov{h^{-1}}$, where $g'\in E_{X_{(\ov{s_g})}/S_{(\ov{s_g})}}(Q)$ and $\ov{g'}$ is the image of $g'$ under $\pi_1(X_{(\ov{s})}, b)\to \pi_1(X, a)\to Q$. By assumption, $\{s^{\ast}(\CF_{i})\}_{i\in I}$ are wildly compatible systems for all geometric points $\ov{s}\to S$. Finally $(\Tr^{\text{Br}}_{\CF_i}(g))_{i\in I}=(\Tr^{\text{Br}}_{\CF_i}(\ov{g}))_{i\in I}=(\Tr^{\text{Br}}_{\CF_i}(\ov{h}\ov{g'}\ov{h^{-1}}))_{i\in I}=(\Tr^{\text{Br}}_{\CF_i}(\ov{g'}))_{i\in I}=(\Tr^{\text{Br}}_{s^{\ast}\CF_i}(g'))_{i\in I}\in V$. So $\{\CF_{i}\}_{i\in I}$ is also a wildly compatible system.

\item The proof is directly from the result of Corollary \ref{myrealfunctorial}, (ii). \qedhere
\eenum
\epf

\section{Preservation by Six Operations}\label{fun}

\begin{thm}\label{sevenoperators}
For $\p{X/S}$ and $\p{Y/S}$ with $\dim\leq 1$ and a morphism $f\colon X\ra Y$, we have:
\benumr
\item $f^{\ast}: K_{c}(Y, I)\to K_c(X, I)$ induces $\fust: K_{c}(Y/S, I,V)\to K_{c}(X/S, I,V)$;
\item $\flst\, : K_c(X, I)\to K_c(Y, I)$ induces $\flst\;\! : K_{c}(X/S, I,V)\to K_{c}(Y/S, I,V)$;
\item $\fls\,\; : K_c(X, I)\to K_c(Y, I)$ induces $\fls\ : K_{c}(X/S, I,V)\to K_{c}(Y/S, I,V)$;
\item $\fus\,\, : K_c(X,I)\to K_c(Y, I)$ induces $\fus\; : K_{c}(X/S, I,V)\to K_{c}(Y/S, I,V)$;
\item $\dhom: K_c(X, I)\times K_c(X, I)\to K_c(X, I)$ induces
 \[
 \text{$\dhom: K_{c}(X/S, I,V)\times K_{c}(X/S, I,V)\to K_{c}(X/S, I,V)$ when $V$ is a subalgebra;}
 \]
\item $-\otimes-: K_c(X, I)\times K_c(X, I)\to K_c(X, I)$ induces
\[
\text{$-\otimes-: K_{c}(X/S, I,V)\times K_{c}(X/S, I,V)\to K_{c}(X/S, I,V)$ when $V$ is a subalgebra;}
\]
\item $D_X: K_c(X,I)\to K_c(X,I)$ induces $D_X: K_{c}(X/S, I,V)\to K_{c}(X/S, I,V)$.
\eenum
\end{thm}
This theorem also implies results for the three cases in Remark \ref{Compare}, including Vidal's (\cite{VidalII}*{Cor.~0.2}) and Yatagawa's (\cite{Yatagawa}*{Cor.~4.1}).  In \ref{Compare}, note that $V$ are subalgebras for (a) and (c) but not (b).
\begin{proof}
By \Cref{compatiblestrictlocal}, it suffices to prove for the case when $S$ is strictly local. Further $f^{\ast}=\fus$ by \cite{Zhengsixop}*{Cor.~9.5}, $\flst=\fls$ by  \cite{Illusie-Zheng}*{Thm.~3.14} generalizing a theorem of Laumon (\cite{Laumon}). So (i)--(iv) reduces to $\fust$ and $\fls$. By \Cref{finiteset}, we may assume that $I$ is a finite set.

\begin{itemize}
\item[(i)\& (iv)] This proof is similar to \cite{VidalI}*{Prop.~2.3.3} and \cite{Yatagawa}*{Cor.~4.1}. For $(\CF_i)_{i\in I}\in K_c(Y, I)$, we may assume that they are constructible sheaves. There is a common stratification for $\CF_i$: $Y=\coprod_{\sigma}Y_{\sigma}$ such that $(\CF_i)_{i\in I}\in K_{\text{coh}}(Y_{\sigma}/S,I, V)$. Take preimages of $Y_{\sigma}$ and take their stratifications we have $X=\coprod_{\sigma,\mu}X_{\sigma \mu}$ such that $\fust \CF_i|_{X_{\sigma \mu}}$ are locally constant sheaves. By Proposition \ref{hahaha} the assertion follows.

\item[(ii)\&(iii)] Let $(a_{i})_{i\in I}\in K_{c}(X/S, I,V)$. First, we prove the assertion when $f$ is an immersion. If $f$ is an open immersion, then $Y=X\coprod (Y\backslash X)$ and $\fls a_{i}|_X=a_{i}$, $\fls a_{i}|_{Y\backslash X}=0$. If $f$ is a closed immersion, then we have $\flst a_{i}|_X=a_{i}$, $\flst a_{i}|_{Y\backslash X}=0$. By taking a stratification of $Y\backslash X$, we check on each stratum and find that $\{a_{i}\}_{i\in I}$ for a wildly compatible system.

For general cases, we decompose $X=X_{\eta}\coprod X_{s}$, $Y=Y_{\eta}\coprod Y_{s}$ as closed fibre parts and generic fibre parts. Write $f_{\eta}: X_{\eta}\to Y_{\eta}$ and $f_{s}: X_s\to Y_s$. Let $i_{Y_{\eta}}: Y_{\eta}\to Y$ and $i_{Y_s}: Y_s\to Y$ be immersions. By $\fls a_{i}=i_{Y_{\eta}!}f_{\eta !}(a_{i}|_{X_{\eta}})+i_{Y_s!}f_{s!}(a_{i}|_{X_s})$ for $i\in I$ and the proof above for immersions, it suffices to prove for $f_{\eta}$ or $f_{s}$. The assertion is local so it suffices to prove when $Y=\Spec A$ is affine. Further, we take $X=\Spec B$ affine and therefore $Y=\Spec B[x_1, \cdots, x_r]/I$. The morphism $X\to Y$ is decomposed as a closed immersion followed by a series of morphisms of relative dimension one.

When $f$ is a closed immersion the proof has been finished above. Now let $f$ be a morphism of relative dimension one. Since $a_{i}$ are constructible, we may assume that $Y$ is normal connected and $(\fls a_{i})_{i \in I}\in K_{\text{coh}}(Y, I)$. Further we may assume that $X$ is normal connected and $(a_{i})_{i\in I}\in K_{coh}
(X/S,I,V)$.  For any pair of components $(a_{1}, a_{2})$ of an element $(a_{i})_{i\in I}$ in $K_{\text{coh}}(X/S,I,V)$ its associated galois \'etale covering $p: V\to X$ trivializing them with galois group $H$. By shrinking $Y$ to an open dense subscheme if necessary, we may assume that images of this pair of components under $(f\circ p)_{!}$ are in Grothendieck groups of locally constant sheaves. Let $Q$ be a finite quotient of $\pi_1(Y)$ acting on $\fls  a_{1}$, $\fls a_{2}$, $(f\circ p)_{!}\underline{\mathbb F_{\lambda_1}}$ and $(f\circ p)_{!}\underline{\mathbb F_{\lambda_2}}$.

By shrinking $Y$ to an open dense subscheme, we may assume that $E_{Y/S}(Q)=(E_{\tau/ S}(Y))(Q)$, where $\tau$ is the generic point of $Y$. Let $g\in E_{Y/S}$, then by \cite{VidalII}*{Cor.~3.0.5 and Thm.~0.1}, we express the Brauer traces of $a_{\ell_i}$ for $i=1,2$ as the following
$$\Tr^{\text{Br}}_{\fls a_{i}}(g)=\frac{1}{|H|}\sum_{h'\in H',  p_{P'}(h')=g'}\Tr_{R\Gamma_c(V_{\bar{\tau}}, E_{\lambda_i})}(h')\times \Tr^{\text{Br}}_{a_{i}}(p_H(h')),$$
where $P'$ is a pro-$p$-subgroup of $Gal(\bar{\tau}/\tau)$ with image $P$ under $Gal(\bar{\tau}/\tau)\to \pi_1(Y)$ and $H'=P'\times H$, $E_{\lambda_i}$ is a finite extension of $\mathbb Q_{\ell_i}$ such that its integer ring has the residue field $\mathbb F_{\lambda_i}$. For all $g\in E_{X/S}$ we have $(\Tr^{\text{Br}}_{a_{i}}(g))_{i\in I}\in V$. We need to show that for all $g'\in E_{Y/S}$ we have $(\Tr^{\text{Br}}_{\fls a_{i}}(g'))_{i\in I}\in V$. By \cite{VidalII}*{Prop.~2.2.1} and \cite{VidalII}*{Prop.~2.3.1}, the traces $\Tr_{R\Gamma_c(V_{\ov{\tau}},E_{\lambda_i})}(h')$ are integers independent of $i$. Then by \cite{VidalII}*{Cor.~3.0.7 and Thm.~0.1}, if $h'\in H'$ such that $\Tr_{R\Gamma_c(V_{\ov{\tau}}, E_{\lambda_i})}(h')\neq 0$ then $p_{H}(h')$ is in the image of $E_{X/S}$ by $\pi_1(X)\to H$. Thus $(\Tr^{\text{Br}}_{\fls a_{i}}(g))_{i\in I}\in V$ for every $g\in E_{Y/S}$. So we finish the proof.

\item[(v)]
Since $\dhom_X(\CF_i,\CG_i)\isom D_X(\CF_i\otimes D_X(\CG_i))$, the assertion follows from (vi) and (vii).

\item[(vi)] The assertion follows since Brauer trace of a tensor product is a product of Brauer traces.

\item[(vii)] Let $\{\CF_i\}_{i\in I}\in K_{c}(X, I,V)$. Let $j: U\to X$ be an open immersion such that $\CF_i$ are lisse and $U\to S$ is smooth. Let $i: Y=X\backslash U\to X$ be the closed immersion. Then $D_X \CF_i=D_X \jls \just \CF_i + D_X \ilst \iust \CF_i$. It suffices to prove that $D_X\jls \just \CF_i$ form a wildly compatible system. By definition we have $D_X\jls \just \CF_i=\jlst D_U \just \CF_i =\jlst \dhom_U(\just \CF_i, \mathbb{F}_{\lambda_i}(d)[2d])$. By the result of (ii), it suffices to consider $\dhom_U(\just \CF_i, \mathbb{F}_{\lambda_i}(-d)[-2d])$.
 In fact, for a dual module $M^{\vee}$ of $M$, the Brauer trace satisfies $\Tr^{\text{Br}}_{M^{\vee}}(g)=\Tr^{\text{Br}}_{M}(g^{-1})$. Note that as a result of definition, $E_{X/S}$ is symmetric: if $g\in E_{X/S}$, then $g^{-1}\in E_{X/S}$.
By (i), $\just \CF_i$ form a wildly compatible system.  So $(j^{\ast}\CF_i)^{\vee}(-d)[-2d]$ also form a wildly compatible system. Thus $D_U\jls\just \CF_i$ form a compatible system and by Noether's induction we finish the proof.\qedhere
\end{itemize}
\end{proof}
\section{Relation with compatible $\ell$-adic virtual constructible sheaves}
\label{relation}
Let $X$ be a smooth curve over $\mathbb{F}_{p^f}$. Let $\ell_i$ be primes distinct with $p$, where $i\in I$. Let $E_{i}/\mathbb{Q}_{\ell_i}$ be finite extensions of fields with residue fields $\mathbb{F}_{\lambda_i}$ and $\lambda_i$ are integer powers of $\ell_i$. We fix a field $L$ of characteristic zero and embeddings $\{\iota_i: L\hookrightarrow E_{i}\hookrightarrow \ov{\mathbb{Q}_{\ell_i}}\}$. Let $K$ be the function field of $X$. We recall the definition of (strictly) compatible system of $E_i$-virtual sheaves on $X$ (\cite{Fujiwara}*{1.2}). A family of virtual constructible sheaves $\{\CF_i\}_{i\in I}$ where $\CF_i$ are virtual sheaves over $E_{i}$, is called an $L$-compatible system if for all $x\in |X|$, all geometric points $\ov{x}\to x$ and all $g\in \mathrm{Gal}(\kappa(\ov{x})/\kappa(x))$ such that $g$ are integer powers of geometric Frobenius as $F_0^{f}$ , there exist $a_g\in L$ depending on $g$ such that $\Tr(g, ({\CF_i})_{\ov{x}})=\iota_i(a_g)$. Now we find a connection between $E_i$-compatible systems and wildly compatible systems.

Let $\CO$ be a complete discrete valuation ring and let its fraction field be $E$ of characteristic zero and its residue field be $F$ of characteristic $\ell>0$. Let $X$ be a Noetherian and connected scheme. Then the category of lisse $\ell$-adic sheaves on $X$ is equivalent to the category of continuous $\ell$-adic representations of $\pi_1(X,\ov{x})$, where $\ov{x}$ is a geometric point of $X$. Let $K(X,\CO), K(X, E), K(X, F)$ be the Grothendieck groups of constructible sheaves of $\CO, E, F$-modules over $X$ respectively. Let $j^{\ast}: K(X, \CO)\to K(X, E)$ be the ring homomorphism induced by the exact functor
$$\Modc (X, \CO)\to \Modc (X, E), \qquad M\mapsto E\otimes_{\CO}M$$
Let $i^{\ast}: K(X,\CO)\to K(X, F)$ be the ring homomorphism given by the triangulated functor
$$D^b_c(X , \CO)\to D^b_c(X, F),\qquad M\mapsto F\otimes_{\CO}^{\mathbb{L}} M$$
By \cite{Zhengsixop}*{Prop.~9.4} , $j^{\ast}$ is an isomorphism. Thus we obtain the decomposition morphism $d\ce i^{\ast}\circ {j^{\ast}}^{-1}: K(X, E)\to K(X, F)$.

\begin{thm}\label{ladic}
For any $L$-compatible system $\{\CF_i\}_{i\in I}$ of virtual constructible sheaves on a smooth integral scheme $X$ over $\mathbb{F}_{p^f}$, the image $\{d\p{\CF_i}\}_{i\in I}$ is a wildly compatible system of virtual constructible $\mathbb{F}_{\lambda_i}$-sheaves.
\end{thm}
\begin{proof}
It suffices to prove for $\#I=2$ and locally constant sheaves $\CF_1,\CF_2$ on $X$. If they are $L$-compatible, then by definition, for all closed point $x\in |{X}|$, all $\ov{x}\to x$ and all $g\in \mathrm{Gal}(\kappa(\ov{x})/\kappa(x))$ an integer power of $F_{0}^{f}$, there are $a_g\in L$ such that $\Tr(g, (\CF_{i})_{\ov{x}})=\iota_i(a_g)$. Note that $(\CF_i)_{\ov{\eta}}$ give rise to representations $\rho_i: \pi_1(X,\ov{\eta})\to \mathrm{GL}_{n_i}(E_i)$, where $\ov{\eta}$ is the geometric generic point of $X$. By Chebotarev density theorem (\cite{Serrezeta1965}*{Thm.~7}), the sets consisting of the images of conjugacy classes of $\{\mathrm{Frob}_{x}^n\}_{n\in \mathbb{Z}}$ under $\rho_i$ for all unramified $x\in X$ are dense subsets in $\rho_i(\pi_1(X, \ov{\eta}))$ respectively. Therefore for all $g\in E_{X/\Spec(\mathbb{F}_{p^f})}\subset \pi_1(X,\ov{\eta})$, we have $\Tr^{\text{Br}}_{d(\CF_1)}(g)=\iota_1(a_g)$ and $\Tr^{\text{Br}}_{d(\CF_2)}(g)=\iota_2(a_g)$ for $a_g\in L$. This shows that $\{d(\CF_i)\}_{i\in I}$ is a wildly compatible system in the sense of Remark \ref{Compare}, (c).
\end{proof}
\brem
In \cite{Zhengindependence} Weizhe Zheng has defined $L$-compatible systems over excellent Henselian discrete valuation rings. Further in \cite{LuZhengwild} Qing Lu and Weizhe Zheng have proved that Theorem \ref{ladic} can be generalized to this case.
\erem

%\bibliographystyle{habbrv.bst}
%\bibliography{Cf1}

\end{document}